\documentclass[leqno,12pt]{article}

\usepackage{amssymb, amsmath, amsthm}
\usepackage[all]{xy}

\theoremstyle{plain}
\newtheorem{theorem}{Theorem}[section]
\newtheorem*{theorem*}{Theorem}

\newtheorem{proposition}[theorem]{Proposition}
\newtheorem{lemma}[theorem]{Lemma}
\newtheorem{corollary}[theorem]{Corollary}

\theoremstyle{remark}

\theoremstyle{definition}

\newcommand{\AX}{A_0(X)}
\newcommand{\Cf}{\textit{cf.}\;}
\newcommand{\CH}{\operatorname{CH}}
\newcommand{\CHX}{\CH_0(X)}
\newcommand{\E}{\mathcal{E}}

\newcommand{\Ehat}{\wh{E}}

\newcommand{\Fp}{\mathbb{F}_p}
\newcommand{\FK}{F(K)}

\newcommand{\FmK}{F^m(K)}
\newcommand{\FmppK}{F^{m+1}(K)}
\newcommand{\Gal}{\operatorname{Gal}}
\newcommand{\Gm}{\mathbb{G}_{m}}
\newcommand{\GmK}{G^m(K)}
\newcommand{\Gmpp}{G^{m+1}}
\newcommand{\GmpK}{G^{mp}(K)}
\newcommand{\GmtK}{G^{m+t}(K)}
\newcommand{\Gmphi}{G^m(\phi)}
\newcommand{\Gmppphi}{\Gmpp(\phi)}
\newcommand{\Gmhat}{\wh{\mathbb{G}}_{m}}

\newcommand{\gr}{\operatorname{gr}}
\newcommand{\grm}{\gr^m}
\newcommand{\grmp}{\gr^{mp}}
\newcommand{\grmt}{\gr^{m+t}}
\newcommand{\grto}{\gr^{t/(p-1)}}
\newcommand{\grtto}{\gr^{t + t/(p-1)}}
\newcommand{\grmF}{\grm(F)}
\newcommand{\grtoF}{\grto(F)}
\newcommand{\grmG}{\grm(G)}
\newcommand{\grmpG}{\grmp(G)}
\newcommand{\grmtG}{\grmt(G)}
\newcommand{\grttoG}{\grtto(G)}
\newcommand{\grmphi}{\grm(\phi)}

\newcommand{\grttophi}{\grtto(\phi)}
\newcommand{\grd}{\gr(\delta)}

\renewcommand{\Im}{\operatorname{Im}}
\newcommand{\isomto}{\stackrel{\simeq}{\longrightarrow}}
\newcommand{\Kbar}{\,\overline{\!K\!}\,}
\newcommand{\Ker}{\operatorname{Ker}}
\newcommand{\Kt}{K^{\times}}
\newcommand{\m}{\mathfrak{m}}
\newcommand{\mK}{\m_K}
\newcommand{\N}{\mathbb{N}}

\renewcommand{\O}{\mathcal{O}}
\newcommand{\onto}[1]{\stackrel{#1}{\to}}
\newcommand{\OK}{\O_K}
\newcommand{\OKT}{\OK[\![T]\!]}
\newcommand{\OKtimes}{\OK^{\times}}
\newcommand{\OKbar}{\O_{\Kbar}}
\newcommand{\ol}[1]{\overline{#1}}
\newcommand{\phihat}{\wh{\phi}}
\newcommand{\Qp}{\mathbb{Q}_{p}}
\newcommand{\Res}{\mathrm{Res}}

\newcommand{\ur}{\mathrm{ur}}
\newcommand{\ve}{\varepsilon}
\newcommand{\wt}[1]{\widetilde{#1}}
\newcommand{\wh}[1]{\widehat{#1}}
\newcommand{\Z}{\mathbb{Z}}

\newcommand{\Coker}{\operatorname{Coker}}
\newcommand{\grmKq}{\grm \Kq}
\newcommand{\grmkqn}{\grm\kqn}

\newcommand{\Kq}{K_q}
\newcommand{\KqM}{K_q^M}
\newcommand{\kqn}{k_{q,n}}

\newcommand{\KqMK}{\KqM(K)}

\newcommand{\UmKq}{U^m \Kq}
\newcommand{\UmppKq}{U^{m+1}\Kq}
\newcommand{\Umkqn}{U^m\kqn}
\newcommand{\Umppkqn}{U^{m+1}\kqn}

\def\sn{\smallskip\noindent}

\def\ssm{\smallsetminus}

\title{On the cycle map for products of elliptic curves 
over a $p$-adic field}

\author{Toshiro Hiranouchi and Seiji Hirayama}

\begin{document}
\maketitle
\begin{abstract}
We study the Chow group of $0$-cycles 
on the product of elliptic curves over a $p$-adic field. 
For this abelian variety, 
it is decided that 
the structure of the image  
of the Albanese kernel 
by the cycle class map. 
%
\end{abstract}

\section{Introduction}
Let $X = E\times E'$ be the product 
of elliptic curves $E$ and $E'$ defined 
over 
a finite extension $K$ of the $p$-adic field $\Qp$. 
The main objective of this note is to study 
the Chow group $\CHX$ of $0$-cycles on $X$ 
modulo rational equivalence. 
Let $\AX$ be the kernel of 
the degree map $\CHX \to \Z$ 
and $T(X)$ the kernel of the Albanese map $\AX \to X(K)$ 
so called the Albanese kernel for $X$. 
These maps are surjective, and 
we have $\CHX/T(X) \simeq \Z \oplus X(K)$. 
If we assume $p^n$-torsion points $E[p^n]$ and $E'[p^n]$ 
are $K$-rational,  
Mattuck's theorem \cite{Mat55} on $X(K)$ implies 
$\CHX/p^n \simeq (\Z/p^n)^{\oplus (2[K:\Qp] + 5)} \oplus T(X)/p^n$. 
Raskind-Spie\ss\ \cite{RS00} showed 
the injectivity of  
the cycle map  
$\rho: T(X)/p^n \to H^4(X, \Z/p^n(2))$ 
to the \'etale cohomology group of $X$ 
with coefficients $\Z/p^n(2) = \mu_{p^n} \otimes \mu_{p^n}$ 
when $E$ and $E'$ 
have ordinary or split multiplicative reduction. 
Although it is difficult 
to know the kernel of $\rho$ in general 
(the injectivity fails for certain surfaces, see \cite{PS95}, Sect.\ 8), 
one can calculate the structure of its image. 
This is the main theorem of this note: 

\begin{theorem*}[Thm.\ \ref{thm:ell}]
Let $E$ and $E'$ 
be elliptic curves over $K$ 
with 
good or split multiplicative reduction, 
and $E[p^n]$ and $E'[p^n]$ are $K$-rational. 
The structure of the image of $T(X)/p^n$ for $X = E \times E'$ by 
the cycle map $\rho$ is 

\sn
$\mathrm{(i)}$ $\Z/p^n$ 
if both $E$ and $E'$ have 
ordinary or 
split multiplicative reduction.

\sn
$\mathrm{(ii)}$ $\Z/p^n \oplus \Z/p^n$ 
if $E$ and $E'$ have different reduction types.
\end{theorem*}

The same computation works well in 
the remained case: 
Both of $E$ and $E'$ have supersingular reduction. 
The image may be varied 
according to the $p$-th 
coefficients of multiplication $p$ formula 
of the formal completion of the elliptic curves along 
the origin (\Cf Prop.\ \ref{prop:ss}). 
For an arbitrary elliptic curves $E,E'$ over $K$ 
and $X = E\times E'$, 
the base change   
$X' := X \otimes_K K'$
to some sufficiently large extension field $K'$ over $K$ 
satisfies the assumptions in our main theorem above.  
Since the kernel of the multiplication by $p^n$ 
on $\CHX$ is finite (due to Colliot-Th\'el\`ene, \cite{CT93}), 
we have a surjection $\CH_0(X')/p^n \to \CHX/p^n$ 
with finite kernel 
if we admit Raskind and Spie\ss's conjecture (\cite{RS00}, Conj.\ 3.5.4); 
the finiteness of the kernel of 
the cycle map on $X'$ (\Cf \cite{RS00}, Cor.\ 3.5.2). 
Therefore, we limit our consideration 
as in the above theorem. 
The estimation of 
the difference of the image of $T(X)/p^n$ 
and $T(X')/p^n$ by the cycle maps 
is also a problem. 
Murre and Ramakrishnan (\cite{MR09}, Thm.\ A) gave an answer to this problem 
in the case of $n=1$ 
for the self-product $X = E\times E$ 
of an elliptic curve $E$ over $K$ 
with ordinary good reduction. 
In this case, 
they proved that the structure of 
the image is at most $\Z/p$ 
and is exactly $\Z/p$ if and only if 
the definition field $K(E[p])$ over $K$ 
is unramified with the prime to $p$-part of $[K(E[p]):K]$ 
is $\le 2$ and $K$ has a $p$-th root of unity $\zeta_p$. 

The results in our main theorem 
are known by Takemoto \cite{Tak07}
in the case of ordinary reduction 
or split multiplicative reduction. 
So our main interest is in supersingular elliptic curves. 
In Section 2 
we study the image of the Kummer homomorphism 
associated with isogeny of formal groups. 
The main ingredient is 
the structure of the graded quotients 
of a filtration 
on the formal groups (Prop.\ \ref{prop:str}). 
As a special case, 
we obtain 
the structure of the graded quotients
associated with filtration on the multiplicative group 
modulo $p^n$. 
In Appendix, 
we show that the results work also on the Milnor $K$-groups more generally. 
The proof of the main theorem is given in Section 3.

For a discrete valuation field $K$, 
we denote by
$\OK$    the valuation ring of $K$,
$\mK$    the maximal ideal of  $\OK$, 
$k:=\OK/\mK$  the residue field of $\OK$, 
$v_K$  the normalized valuation of $\OK$, 
$\OKtimes$  the group of units in $\OK$, 
$\Kbar$  a fixed separable closure of  $K$ and 
$G_K := \Gal(\Kbar/K)$ the absolute Galois group of $K$.
For an abelian group $A$ and a non-zero integer $m$, 
let $A[m]$ be the kernel and $A/m$ the cokernel 
of the map $m:A\to A$ defined by multiplication by $m$. 

\medskip\noindent
{\it Acknowledgments.} 
The first author thanks 
Takahiro Tsushima 
for his helpful suggestion on   
the theory of canonical subgroup. 
A part of this note was written 
during a stay of the first author at the Duisburg-Essen university. 
He also thanks the institute for its hospitality.

\section{Formal Groups}
\label{sec:gr}
Let $K$ be a 
complete discrete valuation field of characteristic $0$, 
and $k$ its perfect residue field of characteristic $p>0$. 
In this section, 
we decide the image of the Kummer map 
associated with 
an isogeny of formal groups (Thm.\ \ref{thm:main}).
First we recall some basic notions on formal groups from \cite{Fro68}. 
Throughout this section, 
all formal groups are commutative of dimension one. 
Let $F$ be a formal group over 
the valuation ring $\OK$. 
The elements of the maximal ideal $\m_{\Kbar}$ of $\OKbar$ form 
a $G_K$-module denoted by $F(\Kbar)$ 
under the operation $x + y := F(x,y)$. 
Similarly, for a finite extension $L/K$, 
the maximal ideal $\m_L$ forms a subgroup of $F(\Kbar)$ denoted by $F(L)$. 
For an isogeny $\phi:F\to G$ of formal groups defined over $\OK$, 
we regard it as a power series 
$\phi(T) = a_1 T+ a_2 T^2 +\cdots + a_p T^p + \cdots \in \OKT$. 
The coefficient of $T$ in $\phi(T)$ 
is denoted by $D(\phi) := a_1$. 
The {\it height} of $\phi$ is defined to be a positive integer $n$ 
such that $\phi(T) \equiv \psi(T^{p^n})\mod \mK$ 
for some $\psi\in \OKT$ with $v_K(D(\psi)) = 0$ (\Cf \cite{Kaw02}, 2.1). 
It is known that 
the induced homomorphism $F(\Kbar) \to G(\Kbar)$ from the isogeny 
$\phi:F\to G$ is surjective and 
the kernel of $\phi$ 
(= the kernel of the homomorphism $F(\Kbar) \to G(\Kbar)$ 
induced by $\phi$) is a finite group 
of order $p^{n}$, 
where $n$ is the height of $\phi$. 
For any integer $m\ge 1$, $F^m(K)$ is the subgroup of $F(K)$ 
consisting of the set $\mK^m$.
Fix  a uniformizer $\pi$ of $K$. 
For any $m\ge 1$, we have an isomorphism 
\begin{equation}
 \label{eq:q=1}
  \rho: k \isomto \grmF := \FmK/\FmppK 
\end{equation}
defined by $x\mapsto \wt{x}\pi^m$, 
where $\wt{x} \in \OKtimes$ is a lift of $x\in k \ssm \{0\}$. 
Recall the behavior 
of the operation on the graded quotients  
of raising to 
an isogeny $\phi:F\to G$ with height $1$. 

\begin{lemma}[\cite{Ber74}; \cite{Kaw02}, Lem.\ 2.1.2]
\label{lem:Ber}
Let $\phi(T) := a_1T + a_2T^2 +\cdots$ be an isogeny $F\to G$ 
of formal groups 
defined over $\OK$ with height $1$. 

\sn
$\mathrm{(i)}$ The coefficient $a_p$ is a unit in $\OK$. 

\sn 
$\mathrm{(ii)}$ For $m$ such that $p\nmid m$, we have $a_1 \mid a_m$. 
\end{lemma}
The following lemma is proved essentially 
as same as the case $F = \Gmhat$ the multiplicative group 
(e.g.,\ \cite{FV02}, Chap.\ I, Sect.\ 5). 

\begin{lemma}
\label{lem:grgr}
Let $\phi(T) := a_1T + a_2T^2 +\cdots$ be an isogeny $F\to G$ 
of formal groups 
defined over $\OK$ with height $1$. 
Define $t := v_K(a_1)$ and 
let $a$ be the residue class of $a_1 \pi^{-t}$ 
and $m\ge1$ an integer. 
Then, we have $\phi(\FmK) \subset \GmpK$ for $m \le t/(p-1)$ and 
$\phi(\FmK) \subset \GmtK$ for $m>t/(p-1)$. 
The isogeny $\phi$ induces the following:

\sn
$\mathrm{(i)}$ If $m<t/(p-1)$, the diagram
$$
\xymatrix@C=15mm{
  \grmF \ar[r]^{\phi} & \grmpG \\
  k \ar[u]^{\rho}\ar[r]^{\ol{a}_pC^{-1}} & k\ar[u]_{\rho}
}
$$
is commutative, where $\ol{a}_p \in k$ 
is the residue class of $a_p\in \OKtimes$ 
and $C^{-1}:k\to k$ is ``the inverse Cartier operator''
\footnote{
For the original definition of 
the inverse Cartier operator, 
see (\ref{eq:iCartier}) in Appendix. 
} 
defined by $x\mapsto x^p$. 
The horizontal homomorphisms are bijective. 

\sn
$\mathrm{(ii)}$ If $m = t/(p-1)$ is in $\Z$, the diagram
$$
\xymatrix@C=15mm{
  \grtoF \ar[r]^{\phi} & \grttoG \\
  k \ar[u]^{\rho}\ar[r]^{a + \ol{a}_pC^{-1}} & k\ar[u]_{\rho}
}
$$
is commutative, where 
the bottom map defined by $x \mapsto ax + \ol{a}_p x^p$.

\sn
$\mathrm{(iii)}$ If $m>t/(p-1)$, the diagram
$$
\xymatrix@C=15mm{
  \grmF \ar[r]^{\phi} & \grmtG \\
  k \ar[u]^{\rho}\ar[r]^{a} & k\ar[u]_{\rho}
}
$$
is commutative, where
the bottom map defined by $x \mapsto ax$. 
The horizontal homomorphisms are bijective. 
Furthermore, we have $G^{m+t}(K) \subset \phi F^{m}(K)$.
\end{lemma}
\begin{proof} 
  Take any $u\pi^m \in \FmK$ with $u\in \OKtimes$. 
  From Lemma \ref{lem:Ber}, 
  we have $v_K(\phi(u\pi^m)) \ge \min\{t + m, pm\}$ 
  (the equality holds if $m\neq t/(p-1)$). 
  Moreover, we have 
  $$
    \phi(u\pi^m) \equiv
    \begin{cases}
      \ol{a}_pu^p\pi^{mp}  \mod \pi^{mp+1},& \mathrm{if}\ m < t/(p-1), \\
       (a u + \ol{a}_pu^p)\pi^{t + t/(p-1)} \mod \pi^{t + t/(p-1) + 1}, & \mathrm{if}\ m = t/(p-1),\\
       a u\pi^{m+t} \mod \pi^{m+t+1}, &\mathrm{if}\ m > t/(p-1).
    \end{cases}
  $$  
  The assertions except the last one follow from it. 
  Using the completeness of $K$, we obtain
  $G^{m+t}(K) \subset G^{m+t +1}(K) + \phi F^{m}(K) \subset G^{m+t +2}(K) + \phi F^m(K) \subset \cdots \subset \phi F^m(K)$ if $m > t/(p-1)$.
\end{proof}
\begin{corollary}
\label{cor:Kaw}
Let $\phi:F\to G$ be an isogeny of formal groups 
defined over $\OK$ with height $1$. 
Assume $F[\phi]:= \Ker(\phi) \subset \FK$. 
For any non-zero element $x \in F[\phi]$, 
we have $v_K(x) = t/(p-1) \in \Z$. 
The kernel of $\phi:\grtoF \to \grttoG$  
is of order $p$. 
\end{corollary}
\begin{proof}
For any non-zero $x\in F[\phi]$, 
we have $\phi(x) = a_1 x + a_2x^2 + \cdots = 0$. 
Hence $t+ v_K(x) = v_K(a_1 x) = v_K(a_p x^p) = p v_K(x)$ 
and $v_K(x) = t/(p-1)$. 
The kernel of $x \mapsto ax + \ol{a}_px^{p}$ is $\sqrt[p-1]{-a/ \ol{a}_p}\Fp$. 
\end{proof}

The filtration $G^m(\phi)$ on $G(\phi) := G(K)/\phi F(K)$ 
is defined by the image of the filtration $\GmK$. 
For an isogeny $\phi:F\to G$ with height $1$, 
its graded quotients $\grmphi := \Gmphi/\Gmppphi$ 
describe the cokernels of $\phi$ in 
Lemma\ \ref{lem:grgr} as follows:

\begin{lemma}
\label{lem:ex <=pt0 ht1}
Let $\phi:F\to G$ be an isogeny over $\OK$ with height $1$ 
and $t:= v_K(D(\phi))$. 

\sn
$\mathrm{(i)}$  
If $m< t + t/(p-1)$, 
the following sequence 
$$
  0 \to \gr^{m/p}(F) \onto{\phi} \grmG \to \grmphi \to 0
$$
is exact, where $\gr^{x}(F) = 0$ if $x \not\in \Z$ by convention.

\sn
$\mathrm{(ii)}$ 
If $m = t+ t/(p-1)$ is in $\Z$, then
$$
   \gr^{t/(p-1)}(F) \onto{\phi} \grttoG \to \grttophi \to 0
$$
is exact. 

\sn
$\mathrm{(iii)}$ 
If $m>t+ t/(p-1)$, then 
$$
  0 \to \gr^{m-t}(F) \onto{\phi} \grmG \to \grmphi \to 0
$$
is exact.
\end{lemma}
\begin{proof}
  Note that we have 
  $\grmphi \simeq G^m(K)/(\phi F(K) \cap G^m(K) + G^{m+1}(K))$. 
  Consider the case (i), (ii). 
  For any $\phi(x) \in G^m(K)$ with $x\in F(K)$, 
  we have an inequality 
  $v_K(\phi(x)) \ge \min\{t + r, pr\}$, where $r =v_K(x)$
  (the equality holds if $r \neq t/(p-1)$) 
  by Lemma \ref{lem:Ber}. 
  To show the injectivity of $\grmG \to \grmphi$ 
  if $p \nmid m$, 
  it is enough to show 
  $\phi F(K) \cap G^m(K) \subset G^{m+1}(K)$. 
  For any $\phi(x) \in G^m(K) \cap \phi F(K)$, 
  assume $m = v_K(\phi(x))$. 
  By Lemma \ref{lem:Ber} as above, 
  $m = pr$ if $t/(p-1) > r$. 
  Otherwise, $m= pt/(p-1)$. 
  This contradicts to $p\nmid m$.
  Thus $v_K(\phi(x)) > m$ and 
  we obtain $\phi(x) \in G^{m+1}(K)$. 
  In the case of $p \mid m$, 
  Take any $\phi(x) \in G^m(K) \cap \phi F(K)$ 
  with $m = v_K(\phi(x))$. 
  From the above (in)equality, we have $v_K(x)  = m/p$. 
  The rest of the assertions follows from it. 
  Next we consider the case (iii). 
  For any $\phi(x) \in G^m(K) \cap \phi F(K)$ 
  with $m = v_K(\phi(x))$. 
  If $t/(p-1) > r$ then $m = pr < pt/(p-1)$ and 
  this contradicts to $m > pt/(p-1)$.
  Otherwise $m \ge t +r$. 
  Hence $r \le m -t$ and thus $x \in F^{m-t}(K)$. 
\end{proof}

Recall that a perfect field 
is said to be {\it quasi-finite}\/ 
if its absolute Galois group is isomorphic to $\widehat{\Z}$ 
(\Cf \cite{Ser68}, Chap.\ XIII, Sect.\ 2). 

\begin{corollary}[\Cf \cite{Ber74}, Lem.\ 1.1.2; \cite{Kaw02}, Lem.\ 2.1.3]
\label{lem:ht1}
Let $\phi(T) := a_1T + a_2T^2 +\cdots$ be an isogeny $F\to G$ 
of formal groups 
defined over $\OK$ with height $1$. 
Assume $F[\phi] \subset F(K)$. 
Define $t := v_K(a_1)$ and 
let $m\ge1$ be an integer. 

\sn
$\mathrm{(i)}$ 
If $m<t + t/(p-1)$, we have
$$
\grmphi \simeq   \begin{cases}
  k ,& \mathit{if}\ p\nmid m,\\
  0 , & \mathit{if}\ p\mid m.
  \end{cases}
$$

\sn
$\mathrm{(ii)}$ 
If $m = t + t/(p-1)$, we have 
$\grttophi \simeq k/(a +\ol{a}_pC^{-1})k$, 
where  $a$  is the residue class of $a_1 \pi^{-t}$. 
If we further assume that $k$ is separably closed, then 
$\grttophi = 0$. 
If $k$ is quasi-finite, then $\grttophi \simeq  \Z/p\Z$.

\sn
$\mathrm{(iii)}$ 
If $m>t+t/(p-1)$, 
we have $G^m(\phi) = 0$. In particular, $\grmphi = 0$.
\end{corollary}
\begin{proof}
The proof below is cited from \cite{Kaw02}.  
The assertions follow from 
Lemmas \ref{lem:grgr} and \ref{lem:ex <=pt0 ht1}. 
If $k$ is quasi-finite, then 
the homomorphism 
$\phi:\gr^{t/(p-1)}(F) \to \gr^{t + t/(p-1)}(G)$ 
is extended to $\phi:\ol{k} \to \ol{k}$. 
Since $H^1(k, \ol{k}) = 1$ 
and $\Ker(\phi) \simeq \Z/p\Z$ as $G_k$-modules, 
we have $k/\phi(k) \simeq H^1(k,\Ker(\phi)) \simeq \Z/p\Z$. 
\end{proof}

Let $\phi:F\to G$ be an isogeny with finite height $n > 1$ 
and {\it assume} $F[\phi]$ is cyclic and 
$F[\phi] \subset F(K)$. 
Let $x_0 \in F(K)$ be the generator of the cyclic group $F[\phi]$. 
The subgroup $pF[\phi] \subset F[\phi]$ 
generated by $[p]x_0$ has order $p^{n-1}$, where 
$[p]$ is the multiplication by $p$ map on $F$. 
From the theorem of Lubin (\cite{Fro68}, Chap.\ IV, Thm.\ 4), 
there exists 
a formal group $G_1 := F/pF[\phi]$ defined over $\OK$ and 
the isogeny $\phi$ factors as $\phi = \phi_1 \circ \psi$, 
where $\psi: F\to G_1$ is an isogeny over $\OK$ such that 
$F[\psi] = pF[\phi]$ (thus $\psi$ is an isogeny with height $n-1$ 
and $\phi_1$ has height 1). 
Note that the kernel $G_1[\phi_1]$ is generated by $\psi(x_0)$. 
From the following lemma, the structure of 
$\grmphi$ is obtained from that in the case of height $1$ 
(Cor.\ \ref{lem:ht1}). 

\begin{lemma}
\label{lem:exact}
Put $t_1:= v_K(D(\phi_1))$. 

\sn
$\mathrm{(i)}$ If $m < t_1+ t_1/(p-1)$, 
the sequence
$$
  0 \to \gr^{m/p}(\psi) \onto{\phi_1} \grmphi \to \grm(\phi_1) \to 0
$$
is exact, where $\gr^{x}(\psi) = 0$ if $x\not\in\Z$ by convention.

\sn
$\mathrm{(ii)}$ If $m = t_1+ t_1/(p-1)$, 
the sequence 
$$
  \gr^{t_1/(p-1)}(\psi) \onto{\phi_1} \gr^{t_1+ t_1/(p-1)}(\phi) 
    \to \gr^{t_1+ t_1/(p-1)}\gr(\phi_1) \to 0
$$
is exact. 

\sn
$\mathrm{(iii)}$ If $m>t_1+t_1/(p-1)$, then 
the sequence
$$
  0 \to \gr^{m-t_1}(\psi) \onto{\phi_1} \grmphi \to \grm(\phi_1) \to 0
$$
is exact. 
In particular, we have an isomorphism $\gr^{m-t_1}(\psi) \simeq \grmphi$.
\end{lemma}
\begin{proof}
  (i) and (ii); $m\le t_1 + t_1/(p-1)$. 
  In the commutative diagram 
  \begin{equation}
    \label{eq:gr}
  \vcenter{
\xymatrix{
  &    \gr^{m/p}(G_1) \ar@{->>}[d] \ar[r]^{\phi_1} & \grmG\ar@{->>}[d] \ar[r] & \grm(\phi_1) \ar[r]\ar@{=}[d] & 0 \\
  &  \gr^{m/p}(\psi) \ar[r]^{\phi_1} & \grmphi \ar[r] & \grm(\phi_1) \ar[r] & 0 &,
    }}
  \end{equation}
  the top horizontal row is exact by Lemma \ref{lem:ex <=pt0 ht1} 
  and the vertical arrows are surjective. 
  We show 
  the injectivity of $\phi_1:\gr^{m/p}(\psi) \to \grmphi$ 
  when $m < t_1+ t_1/(p-1)$ and $m\mid p$. 
  In this case, the map $\phi_1:\gr^{m/p}(G_1) \to \grmG$ in (\ref{eq:gr}) 
  is injective. 
  Thus, it is enough to show 
  the surjectivity of 
  $\phi_1:G_1^{m/p}(K) \cap \psi F(K)/G_1^{m/p+1}(K)\cap \psi F(K) 
  \to G^{m}(K) \cap \phi F(K)/G^{m+1}(K)\cap \phi F(K)$.  
  For any $\phi(x) = \phi_1 \circ \psi(x)$ in $G^{m}(K) \cap \phi F(K)/G^{m+1}(K)$,   
  there exists $y \in G_1^{m/p}(K)$ 
  such that $\phi_1(y) = \phi(x)$ by Lemma \ref{lem:ex <=pt0 ht1}. 
  Hence, we obtain $y = \psi(x) \in G_1^{m/p}(K) \cap \psi F(K)/G_1^{m/p + 1}(K) \cap \psi F(K)$.
  The assertion in (iii) follows from the similar argument as above.  
\end{proof}

Inductively,
one can find isogenies 
$\phi_{i}:G_i \to G_{i-1}$ with height $1$ 
such that $\phi = \phi_1 \circ \cdots \circ \phi_n$ 
and $G_i = F/p^{i}F[\phi]$, 
where $p^iF[\phi]$ is the subgroup of $F[\phi]$ 
generated by $[p^i]x_0$ 
(we denoted by $F = G_n$ and $G = G_0$ by convention). 
Define $t_i := v_K(D(\phi_i))$ and 
put $t_0 := 0$. 

\begin{lemma}
  \label{lem:t_i}
  For $1 \le i < n$, we have 
  $t_i \le t_{i+1}$ and $p^{n-i} \mid  t_i$. 
  The equality $t_i = t_{i+1}$ does not hold 
  if the height of $F > 1$. 
\end{lemma}
\begin{proof}
By induction on $n$, 
it is enough to show the case $n=2$; 
$\phi = \phi_1 \circ \phi_2$ has height $2$. 
Recall $[p]x_0 \in F[\phi_2]$ and $\phi_2(x_0) \in G_1[\phi_1]$. 
From Lemma \ref{lem:Ber}, 
we obtain $t_1/(p-1) = v_K(\phi_2(x_0))$ and 
\begin{equation}
\label{eq:ineq}
t_2/(p-1) = v_K([p]x_0) = 
v_K(\wh{\phi}_2\circ \phi_2(x_0)) \ge v_K(\phi_2(x_0)).
\end{equation}
Hence $t_2 \ge t_1$ and $v_K(\phi_2(x_0)) = pv_K(x_0)$. 
From the inequality (\ref{eq:ineq}), 
if the height of $F> 1$, then 
we have $t_i < t_{i+1}$. 
\end{proof}


\begin{proposition}
\label{prop:str}
Let $\phi= \phi_1 \circ \cdots \circ \phi_n:F\to G $ and $t_i$ 
be as above. 
Put $c_i(\phi) : = t_0 + t_1 + \cdots + t_i + t_i/(p-1)$ for $0 \le i \le n$. 

\sn
$\mathrm{(i)}$ 
If $c_i(\phi) < m < c_{i+1}(\phi)$ for some $0\le i <n$, 
then we have 
  $$
    \grmphi \simeq \gr^{m - (t_1+ t_2 + \cdots +t_i)}(\phi_{i+1}\circ \cdots\circ \phi_n) 
    \simeq
                   \begin{cases}
                    k, & \mathit{if}\   p^{n-i}\nmid m, \\ 
                    0, & \mathit{if}\   p^{n-i}\mid m.
                   \end{cases}
  $$

\sn
$\mathrm{(ii)}$ 
If $m = c_{i+1}(\phi)$, for some $0 \le i < n$, we have
  $$
    \gr^{c_{i+1}(\phi)}(\phi) \simeq \gr^{pt_{i+1}/(p-1)}(\phi_{i+1}\circ \cdots \circ\phi_n) \simeq k/(a + \ol{a}_pC^{-1})k,
  $$ 
where $a_p$ is the coefficient of $T^p$ in $\phi_{i+1} \in \OKT$ 
and 
$a$ is the residue class of $D(\phi_{i+1}) \pi^{-t_{i+1}}$.  
If we further assume that $k$ is separably closed, then 
$\gr^{c_{i+1}(\phi)}(\phi) = 0$. 
If $k$ is quasi-finite, then $\gr^{c_{i+1}(\phi)}(\phi) \simeq  \Z/p\Z$.

\sn
$\mathrm{(iii)}$ 
If $m > c_n(\phi)$ then we have $G^m(\phi) = 0$. 
In particular $\grmphi = 0$. 
\end{proposition}
\begin{proof}
  From Lemma \ref{lem:ht1}, we may assume $n >1$. 
  Put $\psi = \phi_2 \circ \cdots \circ \phi_n$.
  First we consider the case $0 < m < c_1(\phi)$ 
  in (i). 
  In the exact sequence (Lem.\ \ref{lem:exact} (i)) 
  $$
  0 \to \gr^{m/p}(\psi) \to \grmphi \to \grm(\phi_1)\to0,
  $$
  the isogeny $\phi_1$ has height $1$ and $\psi$ 
  has height $n-1$. 
  Thus we obtain the structure of $\grmphi$ 
  for $m <c_1(\phi)$ 
  by induction on $n$ and Lemma\ \ref{lem:ht1} (i). 
  If $m = c_1(\phi)$, 
  $$
    \gr^{t_1/(p-1)}(\psi) \to 
    \gr^{t_1 + t_1/(p-1)}(\phi) \to \gr^{t_1 + t_1/(p-1)}(\phi_1)\to 0
  $$
  by Lemma\ \ref{lem:exact} (ii).
  From Lemma \ref{lem:t_i}, 
  we have $p^{n-1}\mid t_1$. 
  By (i) and the case $m <c_1(\phi)$, 
  $\gr^{t_1 + t_1/(p-1)}(\psi) = 0$ and thus 
  $\gr^{t_1 + t_1/(p-1)}(\phi) \simeq \gr^{t_1 + t_1/(p-1)}(\phi_1)$. 
  The assertion follows from Lemma \ref{lem:ht1} (ii). 
  Consider the case $m> c_1(\phi)$ in (i) and (ii). 
  By Lemma \ref{lem:exact} (iii), 
  $\gr^{m-t_1}(\psi) \simeq \grmphi$. 
  From the induction on $n$, 
  the assertions are reduced to the case $m\le c_1(\phi)$. 
\end{proof}

Let $L/K$ be a finite Galois extension with Galois group $H = \Gal(L/K)$. 
Recall that we call $x$ is a {\it jump} for 
the ramification filtration $(H_j)_{j\ge -1}$ 
in the lower numbering (resp.\ $(H^j)_{j\ge -1}$ in the upper numbering) 
of $H$ 
if $H_x \neq H_{x+\ve}$ 
(resp.\ $H^x \neq H^{x+\ve}$) for all $\ve >0$ 
(for definition of the ramification subgroups, see \cite{Ser68}, Chap.\ IV). 
\begin{proposition}
\label{cor:jump}
For $y \in G^m(\phi)\ssm G^{m+1}(\phi)$, 
take $x \in F(\Kbar)$ with $\phi(x) = y$ in $G(\phi)$. 
If $1 \le m < c_1(\phi)$ and $p\nmid m$, then 
the definition field $L = K(x)$ of $x$ over $K$ 
is totally ramified Galois extension of degree $p^n$. 
The jumps of $H := \Gal(L/K)$ in the upper numbering are 
$c_1(\phi) -m,\ldots ,c_n(\phi) -m$. 
In particular, $H^{c_n(\phi) - m} \neq 1$. 
\end{proposition}
\begin{proof}
  For $n = 1$; namely $\phi = \phi_1$ has height $1$, 
  the assertion follows from \cite{Kaw02}, Lemma 2.1.5. 
  For $n>1$, for the isogeny $\phi = \psi \circ \phi_1$ 
  ($\psi = \phi_2 \circ \cdots \circ \phi_n$), 
  we have $y' \in G_1(\Kbar)$ such that $\psi(x) = y'$ 
  and $\phi_1(y') = y$. 
  The isogeny $\phi_1$ has height $1$, 
  the extension $K':= K(y')/K$ is totally ramified extension 
  of degree $p$. 
  The jump is $pt_1/(p-1) -m$. 
  Since $m < c_1(\phi)$ and $p\nmid m$, 
  $v_{K'}(y) = p v_K(y) = v_K(\phi_1(y')) = pv_{K'}(y')$. 
  Hence $v_{K'}(y') = v_K(y) = m$. 
  By induction on $n$, 
  the extension $L/K'$ is totally ramified extension 
  of degree $p^{n-1}$. 
  The jumps of $\Gal(L/K')$ in the lower numbering 
  are $p^2 t_2/(p-1) -m, p^3 t_3/(p-1) -m ,\ldots ,p^n t_n/(p-1)-m$.
%
  Since the ramification subgroups in the lower numbering 
  commutes with subgroups and in the upper numbering 
  commutes with quotients (\cite{Ser68}, Chap.\ IV), 
  the jumps of the ramification subgroups of $H$ 
  in the lower numbering are 
  $p^i t_i/(p-1)-m$ for $1 \le i \le n$. 
  The ramification subgroups $H^s$ of $H$ in the upper numbering 
  is defined by the Herbrand function $\varphi$ of $H$ as 
  $H_{j} = H^{\varphi(j)}$.   
  For the positive integer $m$, we have  
  $\varphi(m) +1 = \sum_{i=0}^m\#(H_i/H_0)$.
  Thus $\varphi(p^it_i/(p-1) -m) = c_i(\phi)-m$. 
\end{proof}

The isogeny 
$[p^n]: \Gmhat \to \Gmhat$ 
defined by multiplication by $p^n$ on 
the multiplicative group $\Gmhat$ has 
the kernel $\Gmhat[p^n] = \mu_{p^n}$ which is 
cyclic of order $p^n$.  
If $K$ contains a $p^n$-th root of unity $\zeta_{p^n}$,  
$\Gmhat[p^n] \subset \Gmhat(K)$. 
Note also 
the filtration $\Gmhat^j([p^n])$  
of $\Gmhat([p^n]) = \Gmhat(K)/[p^n]\Gmhat(K)  \subset \Kt/p^n$ 
associated with $[p^n]$ is  
the image of the higher unit groups $U_K^j = 1+ \m_K^j$ 
in $\Kt/p^n$ which is also dented by $U_n^j$,  
namely, 
$\Gmhat^j([p^n]) = U_n^j := U_K^j/((\Kt)^{p^n} \cap U_K^j)$. 
Put $U_n^0 := \Kt/p^n$ and 
let $\gr(p^n)$ 
be the graded group 
($= \gr k_{1,n}$ in terms of the appendix) 
associated with the filtration $(U_n^m)_{m\ge 0}$; 
\begin{equation}
  \label{eq:grpn}
  \gr(p^n) := \bigoplus_{m\ge 0}\gr^m(p^n),\quad \gr^m(p^n) := U_n^m/U_n^{m+1}.
\end{equation}
The isogeny $[p^n]$ factors as 
$[p^n] = [p] \circ \cdots \circ [p]$ ($n$ times). 
In particular, $c_i := c_i([p^n]) = ie + e_0$, where 
$e := v_K(p)$ and $e_0 := e/(p-1)$. 
Let $\phi:F \to G$ be an isogeny with height $n$ 
as in Proposition \ref{prop:str}. 
Fix an isomorphism $F[\phi] \simeq \mu_{p^n} = \Gmhat[p^n]$.   
The isogeny induces 
the Kummer homomorphism
$\delta:G(\phi) \to H^1(K, F[\phi]) = \Kt/p^n$. 
We compare the filtration $\Gmphi$ and 
the filtration $\Gmhat^j([p^n]) = U^j_n$ on $\Kt/p^n$. 
In the case of height 1 we have the following theorem: 

\begin{theorem}[\cite{Kaw02}, Thm.\ 2.1.6]
  \label{thm:Kaw}
  Let $\phi:F \to G$ be an isogeny defined over $\OK$ of height $1$, 
  and $t:=v_K(D(\phi))$. 
  Assume that $F[\phi] \subset F(K)$ and $\zeta_p \in K$. 
  Then, the Kummer map $\delta$ induces a bijection 
  $\delta: G^m(\phi) \isomto U_1^{pe_0 - pt/(p-1) + m}$ 
  for any $m\ge 1$.
\end{theorem}

We extend the above theorem to the case of height $>1$ 
assuming $\zeta_{p^n} \in K$. 
Let $\phi= \phi_1 \circ \cdots \circ \phi_n:F\to G$ 
and $t_i$ be as in Proposition \ref{prop:str}. 
Put $c_i(\phi) : = t_0 + t_1 + \cdots + t_i + t_i/(p-1)$ for 
$0 \le i \le n$. 
First we show 
$\delta(G^m(\phi)) \subset U_n^{pe_0 - pt_1/(p-1) + m}$ 
for $m < c_1(\phi) = t_1 + t_1/(p-1)$ with $m\nmid p$.  
Let $j$ be the biggest integer such that 
$\delta (\Gmphi) \subset U^j_n$. 
Because of $p\nmid m$,  
$\delta$ induces 
a non-zero homomorphism 
$\grmphi \to \gr^j(p^n)$.
From Lemma \ref{lem:exact} (i), 
we have the following commutative diagram:
$$
  \xymatrix@C=15mm{
    \grmphi \ar[r]^{\simeq}\ar[d]_{\delta} & \grm(\phi_1) \ar[d]^{\delta} \\
    \gr^j(p^n) \ar[r] & \gr^j(p) 
  }
$$
where 
the top horizontal homomorphism is an isomorphism. 
Since the left $\delta$ is non-zero, 
Theorem \ref{thm:Kaw} implies $j = pe_0 - pt_1/(p-1) + m$. 
In particular we obtain $\delta(G(\phi)) \subset U_n^{e + e_0 - pt_1/(p-1) + 1}$. 
Next, 
we show that 
$\delta$ induces a bijection 
$\gr^m(\phi) \isomto \gr^{c_i - c_i(\phi) + m}(p^n)$ 
on the graded groups by induction on $i$, 
where $c_i = ie +e_0$. 
From Proposition \ref{prop:str} and the above observation, 
(although we do not discuss on $m$ with $p \mid m$)  
the map $\delta$ induces 
$\grmphi \simeq \gr^{pe_0- pt_1/(p-1) + m}(p^n)$ 
for any $m < c_1(\phi)$. 
For $m = c_1(\phi)$, 
let $j$ be the biggest integer such that 
$\delta (G^{t_1 + t_1/(p-1)}(\phi)) \subset U^j_n$ 
as above. 
From Lemma \ref{lem:exact} (b), 
we have 
$$
  \xymatrix@C=15mm{
   & \gr^{t_1 + t_1/(p-1)} (\phi) \ar[r]^{\simeq}\ar[d]_{\delta} & \gr^{t_1 + t_1/(p-1)}(\phi_1) \ar[d]^{\delta} \\
   & \gr^j(p^n) \ar[r] & \gr^j(p)& .
  }
$$
Thus $j = e + e_0 = e + e_0 - pt_1/(p-1) + m$. 
We obtain 
\begin{equation}
  \label{eq:<c_1}
  \gr^m(\phi) \isomto \gr^{e + e_0 - pt_1/(p-1) + m}(p^n)
\end{equation}
for $m \le c_1(\phi)$. 
For $c_{i}(\phi) < m \le c_{i+1}(\phi)$ with $i > 1$, 
let $j$ be the biggest integer such that 
$\delta (G^{m}(\phi)) \subset U^j_n$ again. 
From the induction hypothesis, 
$j > c_i = ie + e_0$. 
By Proposition \ref{prop:str} we have the commutative diagram:
$$
  \xymatrix@C=15mm{
   & \gr^{m-(t_1+ t_2 + \cdots + t_i)} (\phi_{i+1} \circ \cdots \circ \phi_n) \ar[r]^-{\simeq} \ar[d]_{\delta} & \gr^{m}(\phi) \ar[d]^{\delta} \\
   & \gr^{j-ie}(p^{n-i}) \ar[r]^{\simeq} & \gr^j(p^n) &. 
  }
$$
Hence $m-(t_1+ t_2 + \cdots + t_i) \le pt_{i+1}/(p-1)$. 
By the argument above (\ref{eq:<c_1}) 
we obtain $j = c_{i+1} - c_{i+1}(\phi) + m$.
From Proposition \ref{prop:str}, 
we obtain the following theorem: 

\begin{theorem}
\label{thm:main}
The image of the Kummer map $\delta:G(\phi) \to \Kt/p^n$ 
is contained in $U^{c_1 - c_1(\phi) + 1}_n$ 
and $\delta$ induces a bijection
$\grmphi \isomto \gr^{c_i - c_i(\phi) + m}(p^n)$, 
for $m$ with $c_{i-1}(\phi) < m \le c_{i}(\phi)$. 
\end{theorem}

\section{Cycle map}
\label{sec:cycle}
Let $K$ be a finite extension field over $\Qp$.  
Let $X = E \times E'$ be the product of 
two elliptic curves $E$ and $E'$ over $K$ 
with $E[p^n]$ and $E'[p^n]$ are $K$-rational. 
The goal of this section is 
to calculate 
the image of 
the Albanese kernel $T(X) := \Ker(\AX \to X(K))$ 
by the cycle map 
$\rho:\AX \to H^4(X, \Z/p^n(2))$. 
From the argument below which is essentially same as 
in the proof of Theorem 4.3 in \cite{Yam05}, 
the study of the image of $T(X)/p^n$ boils down 
to the calculation of the image of 
the Kummer map $\delta :E(K) \to H^1(K,E[p^n])$ 
and the Hilbert symbol: 
The image of $T(X)$ is contained in $H^2(K,E[p^n] \otimes E'[p^n])$ 
a direct summand of the \'etale cohomology group $H^4(X, \Z/p^n(2))$. 
The Albanese kernel $T(X)$ is isomorphic to 
the {\it Somekawa $K$-group} $K(K; E, E')$ 
defined by  
some quotient of $\bigoplus_{K'/K} E (K') \otimes E'(K')$, 
where $K'$ runs through all finite extensions of $K$ 
(for definition, see \cite{Som90}, \cite{RS00}). 
Thus there is a natural surjection  
$\bigoplus_{K'/K} E (K') \otimes E'(K') \to T(X)/p^n$. 
The cycle map also induces the following commutative diagram  
(\Cf Proof of Prop.\ 2.4 in \cite{Yam05}): 
$$
\xymatrix@R=5mm{
  &\displaystyle{\bigoplus_{K'/K} E (K') \otimes E'(K')} \ar@{->>}[dd] \ar[r]^-{\delta \otimes \delta'} & 
  \displaystyle{\bigoplus_{K'/K} H^1(K',E[p^n])\otimes H^1(K',E'[p^n])\ar[d]^{\cup}}\\ 
  && \displaystyle{\bigoplus_{K'/K} H^2(K',E[p^n] \otimes E'[p^n])\ar[d]^{N}} \\ 
  &T(X)/p^n \ar[r]^-{\rho}&  H^2 (K,E[p^n] \otimes E'[p^n]) &, 
}
$$ 
where  
$\delta$ (resp. $\delta'$)  
is the Kummer map $\delta: E(K') \to H^1(K',E[p^n])$ 
(resp. $\delta: E'(K') \to H^1(K',E'[p^n])$),  
$\cup$ is the cup product  
and $N$ is the norm map. 
From the calculation below, 
the image of the cup product does not depend on 
a extension $K'/K$. 
So we consider the case $K'=K$ only.  
If we fix isomorphisms $E[p^n] \simeq \mu_{p^n}\oplus \mu_{p^n}$ 
and $E'[p^n] \simeq \mu_{p^n}\oplus \mu_{p^n}$, 
the cup product is characterized by the Hilbert symbol 
$(\ ,\ )_n:\Kt \times\Kt \to \mu_{p^n}$ as follows 
(\Cf \cite{Ser68}, Chap.\ XIV, Prop.\ 5): 
$$
  \xymatrix@R=5mm{
    & H^1(K,E[p^n])\otimes H^1(K,E'[p^n])\ar[d]_{\simeq} \ar[r]^-{\cup} 
    &H^2(K,E[p^n]\otimes E'[p^n])\ar[d]^{\simeq} \\
   & {(K^{\times}/p^n)}^{\oplus 2} \otimes 
  {(K^{\times}/p^n)}^{\oplus 2}\ar[r]^-{(\ ,\ )_n^{\oplus 4}} & (\mu_{p^n})^{\oplus 4}&.
  }
$$
Recall that the Hilbert symbol is defined by 
$(a,b)_n := \rho_K(a)(\sqrt[p^n]{b})/(\sqrt[p^n]{b})$, 
for $a,b \in \Kt$, 
where $\rho_K:\Kt \to G_K^{\mathrm{ab}}$ is the reciprocity map. 
Recall that the filtration $U^j_n$ 
on $\Kt/p^n$ is defined by 
the image of $U^j_K$ in $\Kt/p^n$ for $j\ge 1$ and 
$U^0_n := \Kt/p^n$.  
Their orders of the image in $\mu_{p^n}$ 
by the Hilbert symbol are calculated as follows: 

\begin{lemma}[\cite{Tak07}, Prop.\ 2.8]
\label{lem:Tak}
Put $c_i := i e + e_0$ for $1 \le i \le n$, $c_0 := 0$, 
and $c_{n+1} := \infty$. 

\sn
$\mathrm{(i)}$  
$\#(U^s_n,\Kt/p^n)_n = p^{n-i}$ 
for $c_i < s \le c_{i+1}$.

\sn
$\mathrm{(ii)}$ If $p\nmid s$, then 
$\#(U^s_n,U^t_n)_n = p^{n-i}$ for $c_i < s + t \le c_{i+1}$. 

\sn
$\mathrm{(iii)}$ If $s,t >0$, $p\mid s$ and $p\mid t$, then 
$\#(U^s_n,U^t_n)_n = p^{n-i}$ for $c_i \le s+t <c_{i+1}$. 
\end{lemma}

A proof of Lemma \ref{lem:Tak} 
is founded in \cite{Tak07}. 
It is proved by direct computation of 
the Herbrand function of the Kummer extension 
$K(\sqrt[p^n]{b})$ over $K$ for some $b\in \Kt$. 
We present another proof 
using the study in the last section. 

\begin{proof}[Proof of Lemma \ref{lem:Tak}]
  (i) 
  From Proposition \ref{prop:str} (or Thm.\ \ref{thm:A.main}), 
  $s > c_n$  if and only if $U^s_n = 1$. 
  Since the symbol $(\ ,\ )_n$ is non-degenerate, 
  the condition $s > c_n$ is equivalent to 
  $(U^s_n, \Kt/p^n)_n = 1$ for any $n$. 
  It is known that $(a,b)_n^p = (a,b)_{n-1}$ for $a,b\in\Kt$ 
  (\Cf \cite{FV02}, Chap.\ IV, (5.1)). 
  Because of $s >c_1$, 
  the multiplication by $p$ map induces 
  $U^{s-e}_{n-1} \simeq U^{s}_n$ (Lem.\ \ref{lem:exact} (c) 
  or Lem.\ \ref{lem:m>pe0}).  
  By induction on $n$ and $i$, 
  $(U^s_n, \Kt/{p^n})_n \subset \mu_{p^{n-i}}$ 
  if and only if $s> c_i$ for any $n$ and $1 \le i \le n$. 
  
  \sn
  (ii) As in the proof of (i), 
  it is enough to show that, for any $n$, 
  $(U^s_n,U^t_n)_n = 1$ if and only if $s+t >c_n$. 
  For $a,b\in \OK$, 
  we have
  \begin{equation}
  \label{eq:BK86 4.1}
  \begin{split}
  (1+a\pi^s, 1 +b\pi^t)_n 
                        &=(1+a\pi^s(1+b\pi^t), -a\pi^s)_n(1+ab\pi^{s+t}, 1+b\pi^t)_n^{-1}\\
                        &=(1+\frac{ab\pi^{s+t}}{1+a\pi^s}, -a\pi^s)_n^{-1}(1+ab\pi^{s+t}, 1+b\pi^t)_n^{-1}.
  \end{split}
  \end{equation}
  Thus $(U^s_n,U^t_n)_n \subset (U_n^{s+t}, \Kt/p^n)_n$ 
  (\Cf \cite{BK86}, Lem.\ 4.1). 
  If we assume  $s+t >c_n$, then  
  $(U^s_n,U^t_n)_n \subset (U_n^{s+t}, \Kt/p^n)_n = 1$ by (i). 
  Conversely, we show $(U^s_n,U^t_n)_n \neq 1$ 
  for $s+t \le c_n$.  
  We may assume $s\ge t$ and $s +t = c_n$ 
  and hence $p\nmid t$. 
  For $n = 1,2$,  
  Proposition\ \ref{cor:jump} says
  $\Gal(K(\sqrt[p^n]{b})/K)^{s} \neq 1$.  
  Since the reciprocity homomorphism 
  $\rho_K:\Kt \to \Gal(K(\sqrt[p^n]{b})/K)$ 
  maps the higher unit group $U^s_K$ onto 
  the ramification subgroup $\Gal(K(\sqrt[p^n]{b})/K)^s$ 
  (\cite{Ser68}, Chap.\ XV, Cor.\ 3), 
  we obtain $(U^s_n,U^t_n)_n \neq 1$. 
  For $n >2$, we have $s > c_1$. 
  Therefore, 
  $(U^{s-e}_{n-1})^p = U^{s}_n$ (Lem.\ \ref{lem:exact} (c) 
  or Lem.\ \ref{lem:m>pe0}).  
  By induction on $n >2$, 
  there exist $a\in U^{s-e}_{n-1}$ and $b\in U^t_{n-1}$ 
  such that $(a,b)_{n-1} \neq 1$. 
  The assertion follows from 
  $(a^p,b)_n = (a,b)^p_{n} = (a,b)_{n-1} \neq 1$.
  
  \sn
  (iii) 
  As in the proof of (ii),  
  it is enough to show that $(U^s_n,U^t_n)_n = 1$ 
  if and only if $s+t \ge c_n$. 
  If $s +t < c_n$, 
  then $(U^s_n,U^t_n)_n \subset (U^{s+1}_n,U^t_n)_n \neq 1$ 
  from (ii). 
  Suppose $s+t \ge c_n$. 
  By $U_n^0= U_n^1$ and (ii) 
  we may assume $s,t \le 1$ and $s+t = c_n$. 
  From (\ref{eq:BK86 4.1}), 
  for $1 + a\pi^s \in U^s_K, 1 +b\pi^t \in U^t_K$, 
  we have 
  $(1 + a\pi^s, 1+ b\pi^t)_n^{-1} = 
  (1+ab\pi^{s+t}/(1+a\pi^s), -a\pi^s)_n(1+ab\pi^{s+t}, 1+b\pi^t)_n$. 
  From (ii) and $p\mid s$, 
  we obtain $(1 + a\pi^s, 1+ b\pi^t)_n =1$. 
\end{proof}
From the above lemma, 
the Hilbert symbol induces a homomorphism of graded groups: 
Let $M^0:= \mu_{p^n}$ and $M^m := \mu_{p^{n-i}}$ for $m> 0$ such that $c_{i} < m\le c_{i+1}$. 
This filtration $(M^m)_{m\ge 0}$ makes $\mu_{p^n}$ a filtered group. 
The associated graded group $\gr(\mu_{p^n})$ is defined by 
$\gr(\mu_{p^n}) := \bigoplus_{m\ge 0}\gr^m(\mu_{p^n})$, where 
$\gr^m(\mu_{p^n}) := M^m/M^{m+1}$.
On the other hand, let $\gr(p^n)$ be the graded group 
associated with the filtration $(U_n^m)_{m\ge 0}$ defined in (\ref{eq:grpn}). 
If $s,t >0$, $p \mid s$ and $p\mid t$, 
the Hilbert symbol gives 
$(U^s_n, U^t_n)_n = M^{s+t+1}$. 
Otherwise $(U^s_n, U^t_n)_n = M^{s+t}$ (Lem.\ \ref{lem:Tak}). 
We modify the structure of 
the graded tensor product $\gr(p^n) \otimes \gr(p^n)$ 
of the grade groups as follows: 
$\gr(p^n \otimes p^n) := \bigoplus_{m\ge 0}\gr^m(p^n \otimes p^n)$, where 
$$
  \gr^m(p^n \otimes p^n) := \bigoplus_{\begin{smallmatrix}m = s+ t,\\ p\nmid s\, \mathrm{or}\, p\nmid t \end{smallmatrix}}\gr^s(p^n) \otimes \gr^t(p^n) \oplus \bigoplus_{\begin{smallmatrix}m = s+ t + 1,\\ p\mid s\, \mathrm{and}\, p\mid t \end{smallmatrix}}\gr^s(p^n) \otimes \gr^t(p^n) ,
$$
The symbol $(\ ,\ )_n:\Kt/p^n \otimes \Kt/p^n \to \mu_{p^n}$ induces 
$\gr(\ ,\ )_n:\gr(p^n\otimes p^n) \to \gr(\mu_{p^n})$. 
For any subgroups $U$ and $U'$ of $\Kt/p^n$, 
the induced graded subgroups $\gr(U) \subset \gr(p^n)$ and 
$\gr(U') \subset \gr(p^n)$ give 
the graded subgroup 
$\gr(U\otimes U) \subset \gr(p^n \otimes p^n)$. 
The order of the image $(U, U)_n$ coincides with 
that of the image of $\gr(U\otimes U)$ by $\gr(\ ,\ )_n$. 
Since the graded quotient $\gr^m(\mu_{p^n})$ is isomorphic to $\Z/p$
if $m = c_i$ for $i$ and $\gr^m(\mu_{p^n}) = 0$ otherwise,  
this order is 
\begin{equation}
\label{eq:ord}
  \# (U, U')_n = p^{\alpha},\quad 
  \alpha := \# \{ i\ |\ \gr^{c_i}(U \otimes U') \neq 1\ \mathrm{for}\ 0< i\le n \}. 
\end{equation}

Next, 
we study the image of the map 
$\delta:E(K) \to H^1(K,E[p^n]) = \Kt/p^n \oplus \Kt/p^n$. 
When $E$ has {\it split multiplicative reduction}, 
the uniformization theorem gives 
$\Kt/q^{\Z} \simeq E(K)$ for some $q \in K$.  
\begin{theorem}[\cite{Yam05}, Lem.\ 4.5]
Let $E$ and $F$ be elliptic curves over $K$ 
which have split multiplicative reduction. 
Let 
$\phi:E\to F$ be an isogeny over $K$ 
of degree $p^n$ with cyclic kernel $E[\phi]$. 
Assume that the kernel $E[\phi]$ of $\phi$ 
and the kernel $F[\phihat]$ of the dual isogeny $\phihat:F\to E$ 
are $K$-rational.
Then, the image of the Kummer map 
$\delta_{\phi}:E(K) \to H^1(K, E[\phi]) = \Kt/p^n$ is 
$$
  \Im(\delta_{\phi}) = \begin{cases}
                \Kt/p^n,&\mathit{if}\  \sqrt[p^n]{q} \not\in E[\phi],\\
                1, & \mathit{if}\ \sqrt[p^n]{q} \in E[\phi].
                \end{cases}
$$
\end{theorem}

We choose an isomorphism $E[p^n] \simeq \mu_{p^n}\oplus \mu_{p^n}$ 
which maps $E[p^n] \supset \Gm[p^n] = \mu_{p^n}$ onto 
the second factor of $\mu_{p^n}\oplus \mu_{p^n}$.
From the above theorem, we have 
\begin{equation}
\label{eq:sm}
  \Im(\delta) = \Kt/p^n \oplus 1
\end{equation}
when $E$ has split multiplicative reduction. 

We assume that $E$ has {\it ordinary good reduction}. 
Let $\E$ be the N\'eron model of $E$ over $\OK$,  
$\wt{E}$ the neutral component of the special fiber of $\E$, and 
$\pi:E(K) = \E(\OK) \to \wt{E}(k)$ the specialization map. 
The group $E(K) = \E(\OK)$ has a filtration 
$E^i(K)$ ($i\ge 0$) defined by 
$E^0(K) := \E(\OK)$, $E^1(K) := \Ker(\pi)$ 
and for $i\ge 1$, 
$E^i(K) := \{(x,y) \in E(K)\ |\ v_K(x) \le -2 i\} \cup \{ \O \}$, 
where $\O$ is the origin on $E$. 
This filtration coincides 
with $\wh{E}^i(K)$ of the formal group $\Ehat(K)$ defined 
in the previous section. 
Choose an isomorphism $E[p^n] \simeq \mu_{p^n}\oplus \mu_{p^n}$ 
which maps $E^1[p^n]$ onto 
the first factor of $\mu_{p^n}\oplus \mu_{p^n}$. 
Let $x_0$ be a generator of $E[p^n]$ 
and $\Phi$ be the subgroup of $E[p^n]$ generated by $x_0$. 
If $x_0 \in E^1[p^n]$, 
the isogeny $\phi :E \to F := E/\Phi$ has 
the cyclic kernel $E[\phi] = E^1[p^n]$. 
Since $E^1(K)$ is isomorphic to 
the formal group $\Ehat(K)$ 
and 
the height of $\Ehat$ (= the height of $[p]$) is $1$, 
the isogeny $\phi:E\to F$ induces $[p^n]:\Ehat \to \Ehat \simeq \wh{F}$. 
The first factor of the image of 
$\delta:E(K) \to H^1(K, E[p^n]) = \Kt/p^n \oplus \Kt/p^n$ 
coincides with the image of 
the Kummer map $\delta^{1}:\Ehat(K) \to H^1(K,\Ehat[p^n]) = \Kt/p^n$. 
By Theorem \ref{thm:main}, 
the image is $U_n^1$. 
On the other hand, 
if $x_0 \not\in E^1[p^n]$, 
the isogeny $\phi :E \to F := E/\Phi$ has 
the kernel $E[\phi] \simeq \wt{E}[p^n]$. 
Hence, the image of $\delta_{\phi}:E(K) \to H^1(K, E[\phi])$ is contained in 
$H^1_{\ur}(K, E[\phi]) 
:= \Ker(\Res:H^1(K,E[\phi]) \to H^1(K^{\ur},E[\phi]))$, 
where 
$K^{\ur}$ is the completion of the maximal unramified extension of $K$ 
and $\Res$ is the restriction map. 
The image of 
$\delta$ 
is contained in $U_n^1 \oplus H^1_{\ur}(K, \mu_{p^n})$. 
Mattuck's theorem \cite{Mat55} says 
$\# E(K)/p^n = ([K:\Qp] + 2)p^{n}$. 
The order of  
$H^1_{\ur}(K, \mu_{p^n}) \simeq H^1(k, \Z/p^n)$ 
is $p^n$ and 
$\#U_n^1 = ([K:\Qp] + 1)p^n$. 
Thus
\begin{equation}
\label{eq:ordinary}
\Im(\delta) = U_n^1 \oplus H^1_{\ur}(K, \mu_{p^n}).
\end{equation}
For the second factor, 
the restriction map $\Res:H^1(K, \mu_{p^n}) \to H^1(K^{\ur}, \mu_{p^n})$ 
induces $\Res^j:U_n^j/U_n^{j+1} \to U_n^{\ur, j}/U_n^{\ur, j}$, 
where $U_n^{\ur,j}$ is the image of $U_{K^{\ur}}^j$ 
in $(K^{\ur})^{\times}/p^n$. 
Proposition \ref{prop:str} implies 
that $\Res^j$ is bijective if $j \neq ie+e_0$ for some $i\le n$ and 
$\Ker(\Res^{c_i}) = U_n^{c_i}/U_n^{c_i+1} = \gr^{c_i}(p^n)$.

Finally, we consider that $E$ has a {\it supersingular good reduction}. 
Let $\Phi$ be a subgroup generated by a generator of $E[p^n]$ 
and we denote by $\phi:E\to F := E/\Phi$ the induced isogeny.  
The first factor of the image of 
$\delta:E(K) \to H^1(K,E[p^n]) = \Kt/p^n \oplus \Kt/p^n$ 
is the image of the Kummer map 
$\delta_{\phi}:F(K) \to H^1(K,E[\phi]) = \Kt/p^n$  
and another one is 
the image of the Kummer map $\delta_{\wh{\phi}}$ 
associated with the dual isogeny $\wh{\phi}$. 
Since the elliptic curve $E$ 
has supersingular reduction, 
$F(K)/\phi E(K)$ is isomorphic to 
${F}^1(K) / (\phi E(K) \cap F^1(K)) \simeq \wh{F}(\phi)$ 
(\cite{Kaw02}, Lem.\ 3.2.3). 
As in the previous section, 
$\phi$ factors as $\phi=\phi_1 \circ \cdots \circ \phi_n$ 
by height $1$ isogenies $\phi_i$. 
The invariants $t_i := D(\phi_i)$ 
satisfy $t_0:= 0 < t_1 < t_2 < \cdots <t_n < e$ 
(Lem.\ \ref{lem:t_i}, see also Thm.\ \ref{thm:KL}). 
Theorem \ref{thm:main} 
says that the image of 
$\delta:E(K) \to \Kt/p^n$ is contained 
in $U_n^{e+ e_0 - (t_1 + t_1/(p-1)) + 1}$. 
More precisely, 
one can describe the image 
in terms of the graded groups as follows: 
From Theorem \ref{thm:main}, 
the graded quotient $\gr^m E := E^m(K)/E^{m+1}(K)$ 
maps onto $\gr^{c_i - c_i(\phi) + m}(p^n)$ for $c_{i-1}(\phi) < m \le c_{i}(\phi)$, 
where $c_i(\phi) := t_0 + t_1 + \cdots + t_i + t_i/(p-1)$ 
and $c_i := ie + e_0$. 
Hence $\delta$ induces a surjection 
$$
  \grd:\gr E := \bigoplus_{m\ge 0} \gr^m E \longrightarrow  
\bigoplus_{i = 1}^n \bigoplus_{c_{i-1}(\phi) < m \le c_i(\phi)} \gr^{c_i - c_i(\phi) + m}(p^n).
$$
Similarly, 
the dual isogeny $\wh{\phi}$ is described by 
the dual isogenies $\wh{\phi}_i$ of $\phi_i$  
as $\wh{\phi} = \wh{\phi}_n \circ \cdots \circ \wh{\phi}_1$. 
The invariants 
$\wh{t}_i := D(\wh{\phi}_{n-i+1}) = e - t_{n-i+1}$ 
satisfy $\wh{t}_0 := 0 < \wh{t}_1 < \wh{t}_{2} <\cdots < \wh{t}_n < e$. 
Thus $c_i(\wh{\phi}) = \wh{t}_0 + \wh{t}_1 + \cdots + \wh{t}_i + \wh{t}_i/(p-1)$. 
Summarize the above observations 
in terms of the graded groups, we have:

\begin{theorem}
\label{thm:ell2}
The Kummer map 
$\delta:E(K) \to H^1(K,E[p^n]) = \Kt/p^n \oplus \Kt/p^n$ 
induces $\grd: \gr E \to \gr(p^n) \oplus \gr(p^n)$ 
on graded groups, 
where $\gr(p^n) := \bigoplus_{j\ge 0}\gr^j(p^n)$. 

\sn
$\mathrm{(i)}$ If $E$ has split multiplicative reduction,  
$\Im(\grd) = \gr(p^n) \oplus 1$.

\sn
$\mathrm{(ii)}$ If $E$ has ordinary reduction, 
$$
  \Im(\grd) = \bigoplus_{j\ge 1}\gr^j(p^n) \oplus \bigoplus_{i=1}^n \gr^{c_i}(p^n).
$$

\sn
$\mathrm{(iii)}$ If $E$ has supersingular reduction, then 
$$
 \Im(\grd) = 
   \bigoplus_{i = 1}^n \bigoplus_{c_{i-1}(\phi) < m \le c_i(\phi)} \gr^{c_i - c_i(\phi) + m}(p^n) 
    \oplus 
   \bigoplus_{i = 1}^n \bigoplus_{c_{i-1}(\wh{\phi}) < m \le c_i(\wh{\phi})} \gr^{c_i - c_i(\wh{\phi}) + m}(p^n). 
$$
\end{theorem}

Now we complete the proof of the main theorem. 
\begin{theorem}
\label{thm:ell}
Let $E$ and $E'$ 
be elliptic curves over $K$ 
with (semi-)stable reduction 
and $E[p^n]$ and $E'[p^n]$ are $K$-rational. 
The structure of the image of $T(X)/p^n$ for $X = E \times E'$ by 
the cycle map $\rho$ is 

\sn
$\mathrm{(i)}$ $\Z/p^n$ 
if both $E$ and $E'$ have 
ordinary or 
split multiplicative reduction.

\sn
$\mathrm{(ii)}$ $\Z/p^n \oplus \Z/p^n$ 
if $E$ and $E'$ have different reduction types.
\end{theorem}
\begin{proof}
We denote the subsets of $\N := \Z_{\ge 0}$ 
which indicate the indexes of the graded quotients of $\Im(\grd)$ by 
$M := \{m\ge 0\},\ O := \{m\ge 1\},\ O_{\ur} := \{m = c_i \ |\ 0 < i \le n\}$, 
\begin{align*}
  S &:= \bigcup_{i=1}^n\{ c_i - \frac{pt_i-t_{i-1}}{p-1} < m \le c_i\},\ \mathrm{and}\\
  \wh{S} &:= \bigcup_{i=1}^n\{ c_{i-1} + \frac{pt_{n-i+1}-t_{n-i+2}}{p-1} < m \le c_i\}, 
\end{align*}
where $t_{n+1} := e$ by convention.
Define 
$$
d_j: E(K) \otimes E'(K) \onto{\delta \otimes \delta'} (\Kt/p^n\otimes \Kt/p^n)^{\oplus 4}  
\onto{\mathrm{pr}_j} \Kt/p^n \otimes \Kt/p^n,
$$
where $\mathrm{pr}_j$ is the $j$-th projection. 
We calculate the order of the image of 
the composition $(\ ,\ )_n \circ d_j:E(K)\otimes E'(K) \to \mu_{p^n}$ 
for each $j$ in the following five cases:

\sn 
(a) Both of $E$ and $E'$ have split multiplicative reduction.\\ 
(b) Both of $E$ and $E'$ have ordinary reduction.\\ 
(c) $E$ has ordinary reduction 
and  $E'$ has split multiplicative reduction.\\
(d) $E$ has supersingular reduction 
and  $E'$ has split multiplicative reduction.\\
(e) $E$ has supersingular reduction 
and  $E'$ has ordinary reduction.


\sn
First we consider the easiest case (a): 
Both of $E$ and $E'$ have split multiplicative reduction. 
From (\ref{eq:sm}), the images of $d_j$ are
$\Kt/p^n \otimes \Kt/p^n$, 
$\Kt/p^n \otimes 1$, $1 \otimes \Kt/p^n$ and $1\otimes 1$. 
By Lemma \ref{lem:Tak}, 
the image of the cycle map is isomorphic to $\Z/p^n$. 

\sn
Case (b): Both of $E$ and $E'$ have ordinary reduction. 
From (\ref{eq:ordinary}), replace the index $j$ if necessity,  
the image of $d_j$ is 
$\Im(d_1) = U_n^1 \otimes U_n^1$, 
$\Im(d_2) = U_n^1 \otimes H^1_{\ur}(K, \mu_{p^n})$, 
$\Im(d_3) = H^1_{\ur}(K,\mu_{p^n})\otimes U_n^1$ and 
$\Im(d_4) = H^1_{\ur}(K, \mu_{p^n}) \otimes H^1_{\ur}(K, \mu_{p^n})$. 
The image of $\Im(d_1)$ by the Hilbert symbol is $\mu_{p^n}$ 
(Lem.\ \ref{lem:Tak}). 
We count the order of the image of $d_2$ in the graded groups.  
A subset $R_i$ of $\N \times \N$ is define by
\begin{equation}
\label{eq:Ri}
  R_i := \{ (s, ie +e_0 -s)\ |\ 0 < s < ie+e_0,\ p\nmid s\} \cup \{ (0,ie+e_0), (ie+e_0 ,0)\}.
\end{equation}
By (\ref{eq:ord}), the order of $\Im(d_2)$ is $p^{\alpha}$, where 
$\alpha = \#\{ i \ |\ (O \times O_{\ur}) \cap R_i \neq \emptyset \}$. 
However, $(O \times O_{\ur} ) \cap R_i = \emptyset$ for all $i$. 
Thus $\#\Im(d_2) = \#\Im(d_3) = 0$. 
Because $O_{\ur} \subset O$, we also obtain $\#\Im(d_4) = 0$. 

\sn
Case (c): 
Assume that $E$ has ordinary reduction 
and  $E'$ has split multiplicative reduction. 
Enough to consider the image of $U^1_n \otimes \Kt/p^n$ 
and $H^1_{\ur}(K,\mu_{p^n}) \otimes \Kt/p^n$ by the Hilbert symbol. 
For the later, the required order is $p^{\alpha}$, 
$\alpha = \#\{ i \ |\ (O_{\ur} \times M) \cap R_i \neq \emptyset \}$. 
Since $O_{\ur} \times O \subset O_{\ur} \times M$, 
$\alpha = n$ from (b). 
By $O_{\ur} \times M \subset O \times M$, 
we obtain the order of the image of $U^1_n \otimes \Kt/p^n$  
is also $p^n$.
 
\sn
Case (d): $E$ has supersingular reduction 
and  $E'$ has split multiplicative reduction.\ 
Since $O \times M \subset S \times M$ and 
$O_{\ur} \times M \subset \wh{S} \times M$, 
the image is isomorphic to $\Z/p^n \oplus \Z/p^n$ by (c). 

\sn
Case (e): $E$ has supersingular reduction 
and  $E'$ has ordinary reduction. 
For each $i$, 
$(ie +e_0 -1,1) \in (S \times O) \cap R_i$ and  
$(ie +e_0 -1,1) \in (\wh{S} \times O) \cap R_i$.  
On the other hand 
$S \times O_{\ur}, \wh{S} \times O_{\ur} \subset O \times O_{\ur}$.  
Thus the image is isomorphic to $\Z/p^n \oplus \Z/p^n$. 
\end{proof}

%


When both of $E$ and $E'$ have supersingular reduction also, 
the computation of the image 
$\rho(T(X)/p^n)$ is done by 
the similar argument as in the proof of the above theorem. 
The results depend on the invariants 
$t_1 < t_2 < \cdots < t_n$ 
associated with the formal group $\wh{E}$ 
and $t_1' < t_2' < \cdots < t_n'$ 
associated with $\wh{E}'$ 
defined in the previous section. 
These invariants are calculated from 
the theory of the {\it canonical subgroup} due to Katz-Lubin. 
The canonical subgroup $H(E)$ 
of an elliptic curve $E$ (when it exists) 
is a distinguished subgroup of order $p$ in $\Ehat[p]$ 
which play the crucial role in the theory 
of overconvergent modular forms. 

\begin{theorem}[\cite{Kat73}, Thm.\ 3.10.7; \cite{Buz03}, Thm.\ 3.3]
\label{thm:KL}
Let $E$ be an elliptic curve over $K$ with supersingular reduction. 
Let $a(\Ehat)$ be the $p$-th coefficient of multiplication $p$ 
formula $[p](T)$ of the formal group $\Ehat$. 

\sn
$\mathrm{(i)}$ If $v_K(a(\Ehat)) <pe/(p+1)$, 
then the canonical subgroup $H(\Ehat) \subset \Ehat[p]$ exists. 
For any non-zero $x \in \Ehat[p]$, 
$$
  v_K(x) = \begin{cases}
           \frac{e-v_K(a(\Ehat))}{p-1},& \mathit{if}\ x \in H(\Ehat),\\
           \frac{v_K(a(\Ehat))}{p^2-p},& \mathit{otherwise}.
           \end{cases}  
$$
For a subgroup $H\neq H(\Ehat)$ of $\Ehat[p]$, 
$v_K(a(\Ehat/H)) = v_K(a(\Ehat))/p$ and the canonical subgroup $H(\Ehat/H)$ 
of the quotient $\Ehat/H$ 
is the canonical image of $\Ehat[p]$ in $\Ehat/H$. 
Moreover, 

\sn
$\mathrm{(a)}$ If $v_K(a(\Ehat)) <e/(p+1)$, then 
$v_K(a(\Ehat/H(\Ehat))) = p v_K(a(\Ehat))$. 
The canonical image of $\Ehat[p]$ in $\Ehat/H(\Ehat)$ 
is not the canonical subgroup of $\Ehat/H(\Ehat)$. 

\sn
$\mathrm{(b)}$ If $v_K(a(\Ehat)) = e/(p+1)$, then 
$v_K(a(\Ehat/H(\Ehat))) \ge pe/(p+1)$. 

\sn
$\mathrm{(c)}$ If $e/(p+1) < v_K(a(\Ehat)) < pe/(p+1)$, then 
$v_K(a(\Ehat/H(\Ehat))) = e -v_K(a(\Ehat))$ 
and the canonical subgroup of $\Ehat/H(\Ehat)$ is 
$H(\Ehat/H(\Ehat)) = \Ehat[p]/H$. 

\sn
$\mathrm{(ii)}$ If $v_K(a(\Ehat)) \ge pe/(p+1)$, 
then $v_K(x) = e/(p^2-1)$ for any non-zero $x \in \Ehat[p]$. 
For any subgroup $H$ of $\Ehat[p]$,  
$v_K(a(\Ehat/H)) = e/(p+1)$ and the canonical subgroup of the quotient 
$\Ehat/H$ is the image of $\Ehat[p]$ in $\Ehat/H$. 
\end{theorem}

For $n = 1$,  
let $x_0$ be a generator of $\Ehat[p]$ 
such that $v_K(x_0) = \max\{ v_K(x)\ |\ 0 \neq x \in \Ehat[p]\}$. 
Let $\Phi$ be the subgroup of $\Ehat[p]$ 
generated by $x_0$. 
The induced isogeny $\phi:\Ehat \to \Ehat/\Phi$ 
has height $1$. 
If $v_K(a(\Ehat)) < pe/(p+1)$, 
$\Phi= \Ehat[\phi]$ is the canonical subgroup. 
Thus from Corollary \ref{cor:Kaw} and Theorem \ref{thm:KL}, 
we have $t_1/(p-1) = v_K(x_0) = e_0 - v_K(a(\Ehat))/(p-1)$. 
If $v_K(a(\Ehat)) \ge pe/(p+1)$, 
$t_1/(p-1) = v_K(x_0) = e_0/(p+1)$. 
Thus we obtain 
\begin{proposition}
\label{prop:ss}
The structure of the image of $T(X)/p$ for $X = E \times E'$ by 
the cycle map $\rho$ is isomorphic to 

\sn
$\mathrm{(i)}$ $\Z/p \oplus \Z/p$ if $a(\Ehat) \neq a(\Ehat')$ and $a(\Ehat) + a(\Ehat') \neq e_0$,

\sn 
$\mathrm{(ii)}$ $\Z/p$ if $a(\Ehat) = a(\Ehat') \neq e_0/2$, 
or $a(\Ehat) \neq a(\Ehat')$ and $a(\Ehat) + a(\Ehat') = e_0$,

\sn
$\mathrm{(iii)}$ $0$ if $a(\Ehat) = a(\Ehat') = e_0/2$.
\end{proposition}

We conclude this note to give an example: 
Put $p = 5$ and suppose that $E$ is an elliptic curve 
defined by $y^2 = x^3 + ax +b$ over $K$ 
with $v_K(a) \ge 5e/6$ and $v_K(b) = 0$ (\Cf \cite{Ser72}, Sect.\ 1.11). 
Let us consider the self-product of the elliptic curve 
$X = E \times E$. 
Let $\Phi$ be the subgroup generated by a generator of $E[p^2]$. 
The induced isogeny $\phi:E \to F = E/\Phi$ factors as 
$\phi = \phi_1 \circ \phi_2$, where 
$\phi_i:F_i \to F_{i-1}$, $F_1 = E/pF[\phi]$ and $F_0 = F_2 = E$. 
By Theorem \ref{thm:KL}, we have 
%
%
%
%
%
$t_2/(p-1) = v_K(px_0) = e_0/(p+1)$, 
$v_K(a(\wh{F}_1)) = e/(p+1)$ and 
$t_1/(p-1) = v_K(\phi_2(x_0)) = e_0/p(p+1)$. 
Thus we have $S = (29e_0/6, 5e_0] \cup (41e_0/5, 9e_0]$, 
$\wh{S} = (5e_0/6, 5e_0] \cup (5e_0 ,9e_0]$,
where $(s,t]$ is the subset of $\N$ 
consists of $n \in \N$ with $s < n \le t$ and $p\nmid n$.\ 
It is easy to see 
$R_1 \cap (S \times S) = R_1 \cap (S \times \wh{S}) =\emptyset$ 
and $R_1 \cap (\wh{S} \times \wh{S}) \neq \emptyset$. 
Here, the set $R_i$ is defined in (\ref{eq:Ri}).
If we {\it assume} 
$e_0 > 6$, then 
$R_2 \cap (S \times S) = \emptyset$. 
However, $R_2 \cap (S \times \wh{S}), R_2 \cap (\wh{S} \times \wh{S})$ 
are non-empty. 
We obtain $\rho(T(X)/p^2) \simeq \Z/p^2 \oplus \Z/p \oplus \Z/p$.


\appendix

\section{Filtration on the Milnor $K$-groups}
In higher dimensional local class field theory of Kato and Parshin, 
the Galois group of an abelian extension field on 
a $q$-dimensional local field $K$ is described by 
the Milnor $K$-group $\KqMK$ for $q\ge 1$. 
The information on the ramification 
is related to the natural filtration $\UmKq$ 
which is by definition the subgroup generated by 
$\{1 + \m_K^m, K^{\times}, \ldots ,K^{\times}\}$, 
where $\m_K$ is the maximal ideal of the ring of integers $\OK$. 
So it is important to know the structure of the graded quotients 
$\grmKq := \UmKq/\UmppKq$. 
In this appendix, 
we shall show that 
the results on the graded quotients in Section \ref{sec:gr} 
associated with filtration on the multiplicative group 
(modulo $p^n$) 
work also on the Milnor $K$-groups. 
For a mixed characteristic Henselian discrete valuation field 
(abbreviated as hdvf in the following) 
which contains a $p^n$-th root of unity $\zeta_{p^n}$, 
we determine the graded quotients $\grmkqn$ 
of the filtration of $\kqn := \KqMK/p^n \KqMK$ 
instead of $\grmKq$ 
in terms of differential forms of the residue field.
J.\ Nakamura described $\grmkqn$ 
after determining $\grmKq$ for all $m$ 
when $K$ is absolutely tamely ramified 
{\it i.e.,}\ the case of $(e,p) = 1$ (\cite{Nak00b}, Cor.\ 1.2). 
Although it is easy 
in the case of $q=1$ (as in (\ref{eq:q=1}) in Sect.\ \ref{sec:gr}), 
the structure of $\grmKq$ 
is still unknown in general.  
In particular, 
when $K$ has mixed characteristic 
and (absolutely) wildly ramification, 
it is known only some special cases (\cite{Kur04}, see also \cite{Nak00}). 
However, as in Section \ref{sec:gr}, 
to study $\grmkqn$ 
we use the structure of $\grm\Kq$ only for lower $m$ 
under the assumption $\zeta_{p^n} \in K$  
(In \cite{Kur04}, Kurihara treated a wildly ramified 
field with $\zeta_p \not\in K$). 

Let $K$ be a hdvf of characteristic $0$, 
and $k$ its residue field of characteristic $p>0$. 
Let $e =v_K(p)$ be the absolute ramification index of $K$ and 
$e_0 := e/(p-1)$. 
For $m\ge 1$, let  $\UmKq$ be the subgroup of $\KqMK$ 
defined as above.
Put $U^0K_q = \KqMK$ and $\grmKq := \UmKq/\UmppKq$.
Let $\Omega_k^1 := \Omega^1_{k/\Z}$ be 
the module of absolute K\"ahler differentials  
and $\Omega_k^q$ the $q$-th exterior power of $\Omega^1_k$ over the residue field $k$. 
Define subgroups $B_i^{q}$ and $Z_i^q$ for $i\ge 0$ of 
$\Omega_k^q$ such that 
$0 = B_0^q \subset B_1^q \subset \cdots \subset Z_1^q \subset Z_0^q = \Omega_k^q$
by the relations 
$B_1^q := \Im(d:\Omega_k^{q-1} \to \Omega_k^q)$, 
$Z_1^q := \Ker(d:\Omega_k^{q} \to \Omega_k^{q+1})$, 
$C^{-1}:B_i^{q-1} \isomto B_{i+1}^q/B_1^q$, and  
$C^{-1}:Z_i^q \isomto Z_{i+1}^q/B_1^q$, 
where $C^{-1}:\Omega_k^q \isomto Z_1^q/B_1^q$ 
is the inverse Cartier operator defined by 
\begin{equation}
\label{eq:iCartier}
  x\frac{dy_1}{y_1} \wedge \cdots \wedge \frac{d y_{q}}{{y_{q}}} 
    \mapsto x^p\frac{dy_1}{y_1} \wedge \cdots \wedge \frac{d y_{q}}{{y_{q}}}.
\end{equation}

First we recall the study on $\grmKq$ for $m \le e+e_0$ 
due to Bloch and Kato 
which is an essential tool for our study. 
We fix a prime element $\pi$ of $K$. 
For any $m$, we have a surjective homomorphism 
$\rho_m: \Omega_k^{q-1} \oplus \Omega_k^{q-2} \to \grmKq$ 
defined by 
\begin{align*}
  \left(x\frac{dy_1}{y_1} \wedge \cdots \wedge \frac{d y_{q-1}}{{y_{q-1}}}, 0\right) & \mapsto \{ 1 + \pi^m \wt{x}, \wt{y}_1, \ldots \wt{y}_{q-1}\},\\
  \left(0, x\frac{dy_1}{y_1} \wedge \cdots \wedge \frac{d y_{q-2}}{{y_{q-2}}}\right) & \mapsto \{ 1 + \pi^m \wt{x}, \wt{y}_1, \ldots \wt{y}_{q-2}, \pi\},
\end{align*}
where $\wt{x}$ and $\wt{y}_i$ are liftings of $x$ and $y_i$. 
Note that the map $\rho_m$ depends on a choice of $\pi$. 
Using this homomorphism, 
one can obtain the structure of the graded quotients 
$\grmKq$ for any $m \le e+ e_0$ (\cite{BK86}, see also \cite{Nak00}). 
%
%
Next we define the filtration 
$\Umkqn$ on $\kqn= \KqMK/p^n \KqMK$, 
by the image of the filtration $\UmKq$ on $\kqn$. 
Our objective is to study the structure of its graded quotient 
$\grmkqn := \Umkqn/\Umppkqn$. 
From the following lemma, we can investigate $\grmkqn$ for $m>e+e_0$ 
by its structure for $m \le e+e_0$. 

\begin{lemma}
\label{lem:m>pe0}
For $n>1$ and $m>e+e_0$, 
the multiplication by $p$ induces a surjective homomorphism
$p:U^{m-e}k_{q,n-1} \to U^m\kqn$.
If we further assume $\zeta_{p^n} \in K$, 
then the map $p$ is bijective. 
\end{lemma}
\begin{proof}
  The surjectivity follows from the 
  surjectivity of $p: U_K^{m-e} \to U_K^m$ (Lem.\ \ref{lem:grgr}). 
  To show the injectivity, 
  for $x \in U^{m-e}K_q$  
  we assume that $px = p^nx'$ is in $p^n \KqMK \cap \UmKq$ 
  for some $x'\in \KqMK$. 
  Thus $x - p^{n-1}x'$ 
  is in the kernel of the multiplication by $p$ on $\KqMK$. 
  It is known its kernel $= \{\zeta_p\}K_{q-1}^M(K)$.  
  This fact was so called Tate's conjecture.  
  It is a corollary of the Milnor-Bloch-Kato conjecture 
  (due to Suslin, \Cf \cite{Izh00}, Sect.\ 2.4), 
  now is a theorem of Voevodsky, Rost, and Weibel (\cite{Wei09}). 
  Hence, for any $i$ and $y\in K_{q-1}^M(K)$, we have 
  $\{\zeta_{p}^i,y\} = p^{n-1}\{\zeta_{p^n}^i,y\}$.
  Thus we have $x \in p^{n-1}\KqMK$. 
\end{proof}

We determine $\grmkqn$ for any $m$ and $n$ as follows. 

\begin{theorem}
\label{thm:A.main}
We assume $\zeta_{p^n} \in K$. 
Let $m$ and $n$ be positive integers 
and $s$ the integer such that $m = p^sm'$, $(m',p)  = 1$. 
Put $c_i := ie + e_0$ for $i \ge 1$ and $c_0: = 0$.

\sn
$\mathrm{(i)}$ 
If $c_i <  m < c_{i+1}$ for some $0 \le i < n$, 
we have 
$$
\grmkqn \simeq 
\begin{cases}
\Coker(\theta: \Omega_k^{q-2} \to \Omega_k^{q-1}/B_s^{q-1} \oplus \Omega_k^{q-2}/B_s^{q-2}),  &\text{if}\ n-i > s,\\
\Omega_k^{q-1}/Z_{n-i}^{q-1} \oplus \Omega_k^{q-2}/Z_{n-i}^{q-2}, &\text{if}\ n-i \le s, 
\end{cases}
$$ 
where $\theta$ is defined by  
$\omega \mapsto (C^{-s}d\omega, (-1)^q (m-ie)/p^s C^{-s}\omega)$. 

%

\sn
$\mathrm{(ii)}$ 
If $m = c_i$ for some $0 < i \le n$, 
$$
  \gr^{ie+e_0}\kqn \simeq (\Omega_k^{q-1}/(1+aC)Z^{q-1}_{n-i}) \oplus (\Omega_k^{q-2}/(1+aC)Z_{n-i}^{q-2}),
$$
where $C$ is the Cartier operator defined by 
$$
  x^p \frac{dy_1}{y_1} \wedge \cdots \wedge \frac{d y_{q-1}}{{y_{q-1}}} 
    \mapsto x\frac{dy_1}{y_1} \wedge \cdots \wedge \frac{d y_{q-1}}{{y_{q-1}}}. 
$$

%
%
%

\sn
$\mathrm{(iii)}$ 
If $m>c_n$, then $\Umkqn = 0$. 
\end{theorem}

Note that the assertion of the case $m \le e+ e_0$ in the above theorem is 
due to Bloch-Kato (\cite{BK86}, Rem.\ 4.8). 

\begin{proof}[Proof of Thm.\ \ref{thm:A.main}]
As noted above, 
the assertion for $m \le e+ e_0 =c_1$ 
is known. 
It is known also $U^mk_{q,1} = 0$  for $m > e+e_0$ (\cite{BK86}, Lem.\ 5.1 (i)). 
So we assume $m> e + e_0$ and $n>1$.
Thus, for such $m$, we have an isomorphism
$\gr^{m-e}k_{q,n-1} \onto{p} \grmkqn$ 
from the above lemma. 
By induction on $n$, 
we obtain the assertions. 
\end{proof}

\begin{corollary}
If $k$ is separably closed 
{\rm (}we do not need the assumption $\zeta_{p^n} \in K${\rm )}, then 
$\gr^{ie+e_0}\kqn =0$ for $i\ge 1$. 
\end{corollary}
\begin{proof}
The assertion follows from the fact $\gr^{e+e_0}k_{q,1} = 0$ (\cite{BK86}, Lem.\ 5.1 (ii)), 
Lemma \ref{lem:m>pe0}, 
and the induction on $n$. 
\end{proof}



\providecommand{\bysame}{\leavevmode\hbox to3em{\hrulefill}\thinspace}
\providecommand{\href}[2]{#2}

\noindent
Toshiro Hiranouchi \\
Department of Mathematics, Graduate School of Science, Hiroshima University\\
1-3-1 Kagamiyama, Higashi-Hiroshima, 739-8526 Japan\\
Email address: {\tt hira@hiroshima-u.ac.jp}

\medskip\noindent
Seiji Hirayama \\
Faculty of Mathematics, Kyushu University \\
744, Motooka, Nishi-ku, Fukuoka, 819-0395 Japan\\
Email address: {\tt s-hirayama@math.kyushu-u.ac.jp}

\end{document}